\newif\ifcdc

\ifcdc
\documentclass[letterpaper,10pt,conference]{ieeeconf}
\IEEEoverridecommandlockouts
\else
\documentclass{article}
\usepackage[centering,margin=1.5cm]{geometry}
\fi

\makeatletter
\let\NAT@parse\undefined
\makeatother
\usepackage[dvipsnames]{xcolor}          
\usepackage[colorlinks=true, citecolor=blue, linkcolor=black, urlcolor=blue]{hyperref}  
\usepackage[bottom,hang,flushmargin]{footmisc}

\usepackage{setspace}

\let\labelindent\relax
\usepackage{lscape} 
\usepackage{caption} 
\usepackage{amsmath,amssymb,amsfonts,amsthm}
\usepackage{comment}
\usepackage{listings}
\usepackage{verbatim}
\usepackage{textcomp}
\usepackage{enumitem}
\usepackage{stmaryrd}
\usepackage{cases,xspace,enumitem}
\usepackage{subcaption}
\usepackage{multirow}
\usepackage{mdframed}
\usepackage{cancel}
\usepackage{svg}
\usepackage{hyperref}
\usepackage[rightcaption]{sidecap}

\usepackage{version}

\DeclareCaptionLabelFormat{figure}{{Figure #2}}
\usepackage[prependcaption,colorinlistoftodos]{todonotes}

\hypersetup{colorlinks=true, linkcolor=blue, breaklinks=true, urlcolor=blue, citecolor=blue}

\newtheorem{definition}{Definition}
\newtheorem{thm}{Theorem}[section]
\newtheorem{lem}{Lemma}[section]

\newtheorem{remark}[thm]{Remark}

\newtheorem*{assumption*}{Assumption}
\newenvironment{rem}{\begin{remark}}{\qed\end{remark}}

\newcommand{\bd}{\begin{definition}} 
\newcommand{\ed}{\end{definition}}
\newcommand{\bt}{\begin{thm}} 
\newcommand{\et}{\end{thm}}

\newcommand{\R}{\mathbb{R}}

\newcommand{\mcC}{\mathcal{C}}

\newcommand{\bi}{\begin{itemize}} 
\newcommand{\ei}{\end{itemize}} 
\newcommand{\bds}{\begin{description}} 
\newcommand{\eds}{\end{description}} 
\newcommand{\beq}{\begin{equation}} 
\newcommand{\eeq}{\end{equation}} 

\newcommand{\norm}[1]{\|#1\|}

\newcommand{\lognorm}[2]{\mu_{#1}(#2)}
\renewcommand{\lognorm}[2]{\mu_{#2}(#1)} 

\newcommand{\subscr}[2]{#1_{\textup{#2}}}

\newcommand{\setdef}[2]{\{#1 \; | \; #2\}}

\newcommand{\map}[3]{#1\colon #2 \rightarrow #3}

\DeclareMathOperator*{\argmin}{arg\,min}

\newcommand{\ds}{\displaystyle}

\newcommand{\e}{\mathrm{e}}

\newcommand{\amin}{\subscr{a}{min}}
\newcommand{\amax}{\subscr{a}{max}}

\newcommand{\0}{\mbox{\fontencoding{U}\fontfamily{bbold}\selectfont0}}

\definecolor{gnred}{RGB}{255,91,89}

\definecolor{gnred1}{RGB}{71,0,0} 
\definecolor{gnred2}{RGB}{117,0,0} 
\definecolor{gnred3}{RGB}{164,0,0} 
\definecolor{gnred4}{RGB}{211,0,0} 
\definecolor{gnred5}{RGB}{255,0,0} 
\definecolor{gnred6}{RGB}{255,42,34} 
\definecolor{gnred7}{RGB}{255,91,89} 

\definecolor{gnblue1}{RGB}{0,36,71}   
\definecolor{gnblue2}{RGB}{0,60,118}  
\definecolor{gnblue3}{RGB}{0,85,164}
\definecolor{gnblue4}{RGB}{0,108,212}
\definecolor{gnblue4}{RGB}{0,108,212}
\definecolor{gnblue5}{RGB}{0,133,255}  
\definecolor{gnblue6}{RGB}{35,156,255} 
\definecolor{gnblue7}{RGB}{88,177,255} 

\definecolor{gnbrown1}{RGB}{71,27,0}  
\definecolor{gnbrown2}{RGB}{117,45,0} 
\definecolor{gnbrown3}{RGB}{164,62,0} 
\definecolor{gnbrown4}{RGB}{211,80,0} 
\definecolor{gnbrown5}{RGB}{255,97,0} 
\definecolor{gnbrown6}{RGB}{255,127,26} 
\definecolor{gnbrown7}{RGB}{255,155,86} 

\usepackage{booktabs} 
\usetikzlibrary{shapes,arrows,positioning,calc}
\usepackage[most]{tcolorbox}

\ifcdc
\includeversion{cdc}
\excludeversion{arxiv}
\else
\excludeversion{cdc}   
\includeversion{arxiv} 
\fi
\makeatletter
\ifcdc
  \renewenvironment{cdc}{\ignorespaces}{\unskip\ignorespacesafterend}
\else
  
\fi
\makeatother

\definecolor{theoryblue}{RGB}{0, 102, 204}
\definecolor{assumptiongreen}{RGB}{0, 102, 51}
\newtcolorbox[auto counter, number within=section]{myassumption}[2][]{%
    colback=assumptiongreen!5,
    colframe=assumptiongreen,
    coltitle=black,              
    fonttitle=\bfseries,
    title={Assumption~\thetcbcounter~(#2)},
    enhanced,
    attach title to upper,
    after title={.\ },
    arc=5pt,
    outer arc=5pt,
    boxrule=0.4pt,
    left=4pt, right=4pt, top=2pt, bottom=2pt,
    drop shadow,
    #1                           
}

\definecolor{defred}{RGB}{188, 108, 107} 
\newtcolorbox[auto counter, number within=section]{mydefinition}[2][]{%
    colback=defred!5,          
    colframe=defred,           
    coltitle=black,             
    fonttitle=\bfseries,
    title={Definition~\thetcbcounter~(#2)},
    enhanced,
    attach title to upper,
    after title={.\ },
    arc=5pt,
    outer arc=5pt,
    boxrule=0.4pt,
    left=4pt, right=4pt, top=2pt, bottom=2pt,
    drop shadow,
    #1                          
}

\definecolor{thmblue}{RGB}{70, 102, 136} 
\newtcolorbox[auto counter, number within=section]{mytheorem}[2][]{%
    colback=thmblue!5,            
    colframe=thmblue,             
    coltitle=black,               
    fonttitle=\bfseries,
    title={Theorem~\thetcbcounter~(#2)},
    enhanced,
    attach title to upper,
    after title={.\ },
    arc=5pt,
    outer arc=5pt,
    boxrule=0.4pt,
    left=4pt, right=4pt, top=2pt, bottom=2pt,
    drop shadow,
    #1                            
}

\newtcolorbox{myassumptionstar}[1]{%
    colback=assumptiongreen!5,
    colframe=assumptiongreen,
    coltitle=black,
    fonttitle=\bfseries,
    title={Assumption~(#1)},    
    enhanced,
    attach title to upper,
    after title={.\ },
    arc=5pt,
    outer arc=5pt,
    boxrule=0.8pt,
    left=5pt, right=5pt, top=5pt, bottom=5pt,
    drop shadow
}

\newcommand\Item[1][]{%
\ifx\relax#1\relax  \item \else \item[#1] \fi
\abovedisplayskip=0pt\abovedisplayshortskip=0pt~\vspace*{-\baselineskip}}

\title{\LARGE \bf 
A Unified Control-Theoretic Framework for Saddle-Point Dynamics in Constrained Optimization
}

\author{Veronica Centorrino$^{1}$, Rawan Hoteit$^{1}$, Efe C. Balta$^{1,2}$, and John Lygeros$^{1}$
\begin{cdc}
\thanks{*This work was supported as a part of NCCR Automation, a National Centre of Competence in Research, funded by the Swiss National Science Foundation (grant number 51NF40\_225155)}
\end{cdc}%
\thanks{$^{1}$ Automatic Control Laboratory (IfA), ETH Zürich, 8092 Zürich, Switzerland. {\tt\small \{vcentorrino, rhoteit, lygeros\}@ethz.ch}}%
\thanks{$^{2}$ Control and Automation Group, inspire AG, 8005 Zürich, Switzerland. {\tt\small \{efe.balta\}@inspire.ch}}%
}

\begin{document}
\maketitle

\begin{cdc}
\thispagestyle{empty}
\pagestyle{empty}
\end{cdc}

\begin{abstract}
This paper studies constrained minimization problems through the lens of feedback control. We introduce a unified control-theoretic framework by showing that a PID feedback law acting on the dual variable induces the \emph{PID saddle-point flow} (PID-SPF), a broad class of saddle-point dynamics associated with the augmented Lagrangian. This framework recovers several classical primal-dual flows as special cases.
We prove that the equilibria of the proposed flow coincide with the stationary points of the original problem.
Our analysis reveals how the feedback gains affect the optimization: integral action enforces constraint satisfaction, proportional action introduces the augmented Lagrangian structure, and derivative action modifies the geometry of the primal dynamics by inducing a state-dependent Riemannian metric.
Moreover, for convex problems with affine constraints, we establish global exponential convergence by leveraging contraction theory for all admissible PID gains, providing in the process explicit bounds on the convergence rate.
Finally, we validate our theoretical results on numerical examples including an application to bilevel optimization.
\end{abstract}
\section{Introduction}
We study equality-constrained optimization problems (OPs) of the form
\begin{equation}
\begin{aligned}
    \min_{x \in \R^n} & f(x)\\
    \text{s.t. } & h(x) = \0_m,
\end{aligned}
\label{eq:eq_constrained}
\end{equation}
where $\map{f}{\R^n}{\R}$ and $\map{h}{\R^n}{\R^m}$ are continuously differentiable functions.
Such problems are ubiquitous in engineering, science, and machine learning, where constraints naturally encode physical laws, or other coupling constraints between decision variables.

A classical approach to solving~\eqref{eq:eq_constrained} relies on primal-dual flows derived from the associated Lagrangian~\cite{KJA-LH-HU:58, GQ-NL:19}. From a control-theoretic perspective, these flows can be viewed as closed-loop dynamical systems where the Lagrange multipliers act as control inputs designed to enforce constraint satisfaction. In this light, the standard primal-dual dynamics results from an integral controller acting on the constraint error.
The recent work~\cite{VC-SMF-SP-DR:25} formalizes this perspective via a control-theoretic framework for equality-constrained problems where the primal dynamics defines the plant, the constraint violation constitutes the system output, and the multiplier is the control input. This viewpoint shifts the focus from algorithm design to feedback control design and raises the question of how different feedback laws affect the resulting optimization dynamics and their geometry.

Motivated by this perspective, we introduce a proportional-integral-derivative (PID) feedback law on the dual variable. This approach leads to a unified class of saddle-point dynamics and provides a systematic interpretation of classical primal-dual flows, including Arrow-Hurwicz-Uzawa, augmented Lagrangian, and projected gradient flows~\cite{KJA-LH-HU:58, DPB:97}.

\textit{Literature review:}
Studying OPs via continuous-time dynamics is a classical problem dating back to~\cite{KJA-LH-HU:58}. This perspective has gained renewed interest thanks to developments from, e.g., online  feedback optimization~\cite{GB-JC-JIP-EDA:22} and the interpretation of optimization algorithms through the lens of feedback control~\cite{FD-ZH-GB-SB-JL-MM:24, AH-ZH-SB-GH-FD:24}. While nonlinear sign-based control has been deployed for projected gradient flows solving linearly-equality constrained convex OPs~\cite{CF-RW:20}, the standard approach for constrained OPs is via primal-dual dynamics~\cite{KJA-LH-HU:58, GQ-NL:19}. 
The use of Lagrange multipliers as feedback controllers has recently been analyzed for smooth and non-smooth OPs with equality~\cite{VC-SMF-SP-DR:25, RZ-AR-JS-NL:25, VC-FR-FB-GR:25} and inequality~\cite{VC-SMF-SP-DR:24b, RZ-AR-JS-NL:25} constraints, as well as for smooth constrained OPs via control barrier functions~\cite{AA-JC:24}.
For the standard setting of full-rank linear constraints and strongly convex, $L$-smooth objectives, global exponential convergence has been established for both I-~\cite{GQ-NL:19} and PI-controlled dynamics~\cite{VC-SMF-SP-DR:25}. The recent work~\cite{JR-SLJ:26} shows that PI control is equivalent to discrete-time primal-dual dynamics on the augmented Lagrangian.
More broadly, there has been a growing interest in leveraging contracting dynamics to solve OPs~\cite{HDN-TLV-KT-JJES:18, AD-VC-AG-GR-FB:23f, VC-FR-FB-GR:25}. This is motivated by the highly ordered transient and asymptotic behavior properties enjoyed by such dynamics~\cite{FB:26-CTDS_cdc}. Recent work has established contractivity for bilinear saddle-point problems in discrete time via operator-theoretic tools~\cite{CD-MB-PDG-JL-FD:24}.

\textit{Contributions:}
We propose a unified control-theoretic framework for continuous-time saddle-point dynamics arising in equality-constrained optimization.
Our approach builds on a recent closed-loop system interpretation in which the primal dynamics are designed so that its equilibria coincide with the Lagrangian stationary points, while the dual variables act as control inputs to regulate the constraint violation to zero. Within this framework, we show that a PID feedback law combined with a suitable change of variables induces a broad class of dynamics associated with the augmented Lagrangian, the \emph{PID-saddle-point flow} (PID-SPF).

Our main result characterizes how the feedback gains affect the resulting dynamics. Specifically, when the derivative gain is zero, the proposed change of variables defines a global diffeomorphism between the closed-loop dynamics in the original variables and a saddle-point flow of the augmented Lagrangian in the transformed coordinates. When the derivative gain is positive, the PID-SPF corresponds to a Riemannian saddle-point flow, where the derivative action induces a state-dependent metric in the primal space.
We conduct a comprehensive convergence analysis for convex problems with affine constraints and strongly convex, smooth objectives. We establish global exponential convergence of the PID-SPF by leveraging contraction theory. In particular, we show that the proposed dynamics is strongly infinitesimally contracting for all admissible gains and derive explicit bounds on the contraction rate.

Finally, we validate our framework on equality-constrained quadratic programs and on a bilevel optimization problem with a strongly convex lower-level whose optimality condition is subject to uncertainty.
\section{Mathematical Preliminaries}
The sets of real and positive real numbers are denoted by $\R$ and $\R_{>0}$, respectively. For $n \in \mathbb{N}$, $\R^n$ denotes the $n$-dimensional Euclidean space. We denote by $\0_n \in \R^n$ the all-zeros vector in $\R^n$, and by $I_n$ the identity matrix in $\R^{n \times n}$. Given $A, B \in \R^{n\times n}$ symmetric, we write $A \preceq B$ (resp. $A \prec B$) if $B-A$ is positive semidefinite (resp. definite). For a symmetric matrix $A$, we let $\subscr{\lambda}{min}(A)$ and $\subscr{\lambda}{max}(A)$ be its minimum and maximum eigenvalue, respectively.
\paragraph{Norms and Logarithmic Norms} We let $\| \cdot \|$ denote both a norm on $\R^n$ and its corresponding induced matrix norm on $\R^{n \times n}$. Given $A \in \R^{n \times n}$ the \emph{logarithmic norm} (lognorm) induced by $\| \cdot \|$ is 
$$\mu(A) := \lim_{h\to 0^+} \frac{\norm{I_n + h A} -1}{h}.$$
Given a symmetric positive-definite matrix $P \in \R^{n\times n}$, we let $\|\cdot\|_{P}$ be the $P$-weighted $\ell_2$-norm $\|x\|_{P} := \sqrt{x^\top P x}$, $x \in \R^n$. The corresponding lognorm is $\mu_{P}(A) =\min\setdef{b \in \R}{PA + A^\top P \preceq 2bP}$~\cite[Lemma~2.7]{FB:26-CTDS_cdc}.
\paragraph{Calculus and Function Classes}
Given $\map{f}{\R^n}{\R}$ twice differentiable, we denote by $\nabla f(x)$ and $\nabla^2 f(x)$ its gradient and its Hessian matrix, respectively.
For $\map{f}{\R^{n+m}}{\R}$, we let $\nabla_x f(x,y) := \frac{\partial f}{\partial x}(x,y)$, be its partial gradient with respect to $x$. 
Given $\map{h}{\R^n}{\R^m}$ continuously differentiable, $J_h(x) \in \R^{m \times n}$ denotes its Jacobian matrix.
Finally, we recall the following standard definitions.
\begin{mydefinition}[label=def:strongly_convex_smooth]{Strongly convex and $L$-smooth maps}
A map $\map{f}{\R^n}{\R}$ is
\begin{enumerate}[label = (\roman*)]
\item \emph{$\rho$-strongly convex} if there exists $\rho > 0$ such that the map $x \mapsto f(x) - \frac{\rho}{2}\|x\|_2^2$ is convex;
\item \emph{$L$-smooth} if it is differentiable and there exists $L >0$ such that $\nabla f$ is $L$-Lipschitz.
\end{enumerate}
\end{mydefinition}
\paragraph{Contraction Theory}
Consider a dynamical system 
\beq
\label{eq:dynamical_system}
\dot{x}(t) = f\bigl(t,x(t)\bigr),
\eeq
where $\map{f}{\R_{\geq 0} \times \R^n}{\R^n}$. We give the following~\cite{FB:26-CTDS_cdc}.
\begin{mydefinition}[label=def:contracting_system]{Contracting dynamics}
Given a norm $\norm{\cdot}$ with associated lognorm $\mu$, a smooth function $\map{f}{\R_{\geq 0} \times \mcC}{\R^n}$, with $\mcC \subseteq \R^n$ $f$-invariant, open and convex, and a \emph{contraction rate} $c >0$, $f$ is $c$-strongly infinitesimally contracting on $\mcC$ if
$$
\mu\bigl(J_f(t, x)\bigr) \leq -c,\quad \text{ for all } x \in \mcC \textup{ and } t\in \R_{\geq0}.
$$
\end{mydefinition}
\smallskip
If $f$ is contracting, then for any two trajectories $x(\cdot)$ and $y(\cdot)$ of~\eqref{eq:dynamical_system} with initial conditions $x_0$ and $y_0$, respectively,
$$
\|x(t) - y(t)\| \leq \e^{-ct}\|x_0 -y_0\|, \textup{ for all } t \geq 0,
$$
i.e., the distance between the two trajectories converges exponentially with rate $c$. One of the main benefits of contraction theory is that, with just a single condition, it ensures global exponential convergence to the unique equilibrium, along with other useful robustness properties. We refer to~\cite{FB:26-CTDS_cdc} for a recent review of these tools.
\section{Problem Formulation}
\label{sec:problem_formulation}
Consider the equality-constrained optimization problem in~\eqref{eq:eq_constrained}, which we rewrite here for convenience
\begin{align*}
    \min_{x \in \R^n} & f(x)\\
    \text{s.t. } & h(x) = \0_m,
\end{align*}
where $\map{f}{\R^n}{\R}$ and $\map{h}{\R^n}{\R^m}$ are continuously differentiable functions. We assume that~\eqref{eq:eq_constrained} admits at least one feasible point $x^\star$, that is, there exists $x^\star\in \R^n$ such that $h(x^\star)=\0_m$. Starting from the control-theoretic perspective proposed in~\cite{VC-SMF-SP-DR:25}, we reinterpret continuous-time saddle-point flows as a closed-loop system comprising the primal dynamics (the plant) and a multiplier feedback controller.

Consider the Lagrangian associated with~\eqref{eq:eq_constrained}, that is the map $\map{L}{\R^n \times \R^m}{\R}$ defined by
$$L(x, \lambda)=f(x)+\lambda^{\top} h(x),$$
where $\lambda \in \R^m$ is the vector of Lagrange multipliers.
\begin{mydefinition}[label=def:stationary_point]{Stationary point of~\eqref{eq:eq_constrained}}
A pair $(x^{\star},\lambda^{\star}) \in \R^{n+m}$ is a \emph{stationary point} of problem~\eqref{eq:eq_constrained} if
\beq
\label{eq:condition_stationary_point}
\nabla f(x^{\star}) + J_h(x^{\star})^\top\lambda^{\star}= \0_n \quad \text{ and } \quad  h(x^{\star}) = \0_m.
\eeq
\end{mydefinition}
\smallskip
Note that, in general nonlinear programs, stationary points that satisfy the first-order necessary conditions~\cite{DPB:97} are only candidates for optimality and are not necessarily minimizers unless additional conditions (e.g., convexity) hold.

Following~\cite{VC-SMF-SP-DR:25}, we interpret the Lagrange multipliers $\lambda(t) \in \R^m$ as external control inputs acting on the gradient flow of the Lagrangian with respect to the primal variable $x$. This leads to the following input-output system
\beq
\label{eq:open_system}
\begin{cases}
    \dot{x}(t)= - \nabla_x L(x(t), \lambda(t)) = - \nabla f(x(t)) - J_h(x(t))^\top\lambda(t) \\
    y(t)= h(x(t)),
\end{cases}
\eeq
with state $x(t) \in \R^n$, input $\lambda(t) \in \R^m$, and output $y(t) \in \R^m$ representing the constraint violation. 
We recall the following result from~\cite[Lemma 1]{VC-SMF-SP-DR:25}, which characterizes the stationary points of~\eqref{eq:eq_constrained} in terms of the equilibria of~\eqref{eq:open_system}.
\begin{lem}[Stationary points of~\eqref{eq:eq_constrained} and equilibria of~\eqref{eq:open_system}]
\label{lem:stat_points_equilibria}
A point $\left(x^\star, \lambda^\star\right) \in \R^{n+m}$ is a stationary point of~\eqref{eq:eq_constrained} if and only if it is an equilibrium of system~\eqref{eq:open_system} with input $\lambda^\star$, satisfying $h(x^{\star}) = \0_m$.
\end{lem}

As a consequence of Lemma~\ref{lem:stat_points_equilibria}, the control objective is to design a feedback law for $\lambda$ that ensures convergence of~\eqref{eq:open_system} to an equilibrium point while regulating the output to zero. Rather than directly analyzing the resulting closed-loop dynamics induced by a specific given controller, in the next section we address the following structural question:
\begin{tcolorbox}[
    colback=theoryblue!5,     
    colframe=theoryblue,      
    coltitle=white,           
    fonttitle=\bfseries, 
    enhanced, 
    arc=5pt,               
    outer arc=5pt,        
    boxrule=0.8pt,
    drop shadow               
]
\centering
\itshape
How does feedback design on the dual variable shape the structure of saddle-point dynamics?
\end{tcolorbox}
\section{A Unified Control-Theoretic Framework for Saddle-Point Flows}
We show that PID feedback laws induce a unified class of saddle-point flows associated with equality-constrained optimization problems. The resulting system is illustrated in Figure~\ref{fig: Block Diagram}.
\begin{figure}
    \centering
    \includegraphics[width=0.45\linewidth]{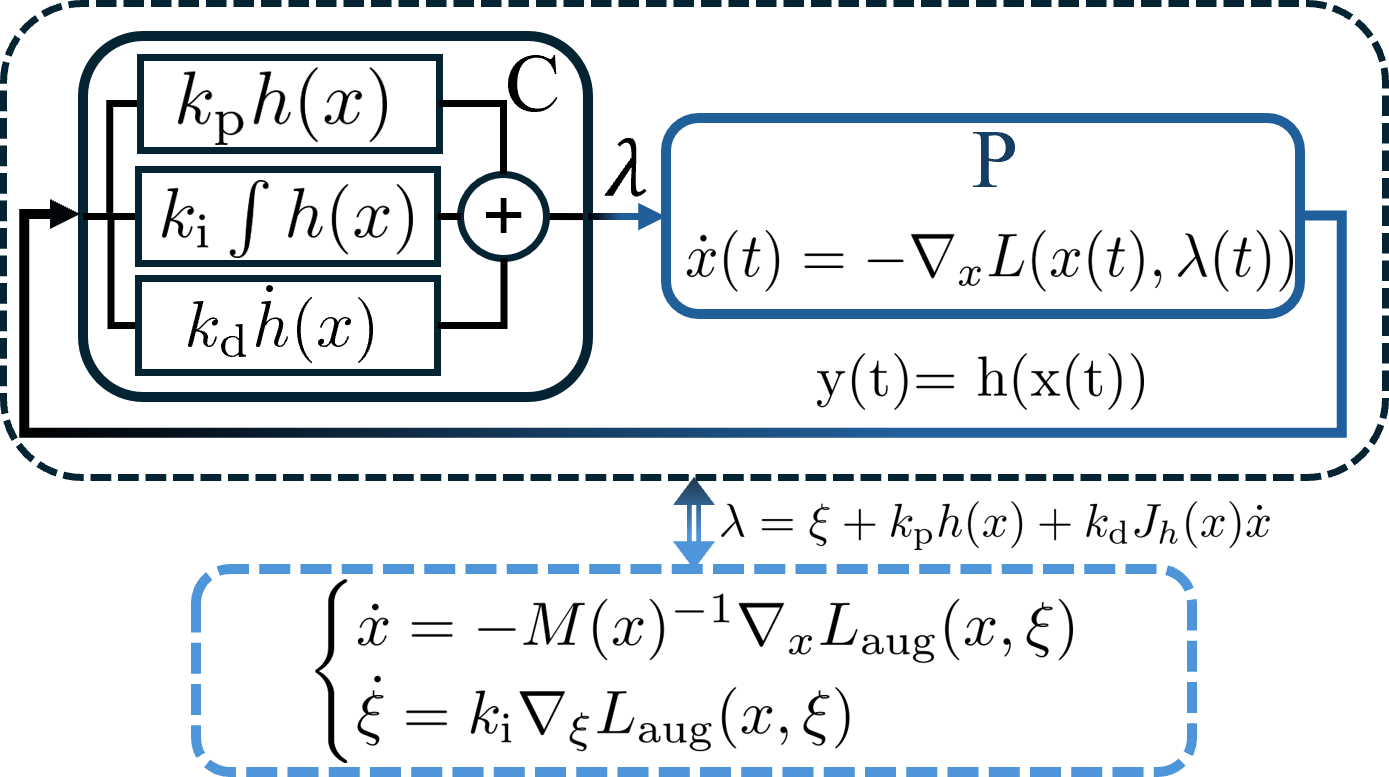}
    \caption{PID-Controlled Primal-Dual Dynamics and the Equivalent Saddle Flows}
    \label{fig: Block Diagram}
\end{figure}
\subsection{PID-Saddle-Point Flow}
While~\cite{VC-SMF-SP-DR:25} considers only a PI controller (resulting in the \emph{PI-controlled multipliers optimization} (PI-CMO) dynamics), we generalize this approach by considering also a derivative term. Let $\lambda(t)$ denote the output of the PID controller
\beq
\label{eq:pid_controller}
\begin{aligned}
\lambda(t) = \subscr{k}{i} \int_0^t y(\tau)d\tau + \subscr{k}{p} y(t) + \subscr{k}{d} \dot{y}(t) = \subscr{k}{i} \int_0^t h(x(\tau))d\tau + \subscr{k}{p} h(x(t)) + \subscr{k}{d} J_h(x(t))\dot{x}(t),
\end{aligned}
\eeq
where $\subscr{k}{i} \in \R_{>0}$ and $\subscr{k}{p},\subscr{k}{d}\in\R_{\ge 0}$ denote the integral, proportional, and derivative gains, respectively. For simplicity, we subsequently drop $(t)$.
As we formalize below, the PID structure in \eqref{eq:pid_controller} induces three distinct mechanisms:
\begin{enumerate}[label=\textup{(\roman*)}]
    \item the \emph{integral term} enforces constraint satisfaction, by accumulating constraint violations and driving them to zero, and thus has to be strictly positive;
    \item the \emph{proportional term} $\subscr{k}{p} h(x)$ modifies the energy landscape by introducing an augmented term to the Lagrangian;
    \item the \emph{derivative term} $\subscr{k}{d} J_h(x)\dot{x}$ alters the geometry of the primal dynamics by inducing a state-dependent metric.
\end{enumerate}
\smallskip
Allowing $\subscr{k}{p}$ and $\subscr{k}{d}$ to vanish recovers the I, PI, and ID controllers.
Differentiating \eqref{eq:pid_controller} and substituting the result into \eqref{eq:open_system} yields the following extension of the PI-CMO dynamics, the PID-CMO dynamics
\begin{equation}
\label{eq:pid_cmo}
\begin{cases}
\dot{x} = - \nabla f(x) - J_h(x)^\top \lambda, \\
\dot{\lambda} = \subscr{k}{i} h(x) + \subscr{k}{p} J_h(x)\dot{x} + \subscr{k}{d} \frac{d}{dt}\big(J_h(x)\dot{x}\big).
\end{cases}
\end{equation}
The derivative term introduces second-order and state-dependent coupling through $\ddot{x}(t)$, which makes the resulting dynamics difficult to analyze. Instead, we propose a change of variables that yields a family of saddle-point flows.
Specifically, let $\xi := \subscr{k}{i} \int_0^t h(x(\tau))d\tau \in\R^m$, then
\beq
\label{eq:controller_general}
\lambda = \xi + \subscr{k}{p} h(x) + \subscr{k}{d} J_h(x)\dot{x},\quad \text{ and } \quad \dot{\xi} = \subscr{k}{i} h(x).
\eeq
Substituting~\eqref{eq:controller_general} into~\eqref{eq:open_system} yields the dynamics
\beq
\label{eq:closed_loop_metric}
\begin{cases}
\dot{x}(t)= -\nabla f(x) - J_h(x)^\top \left(\xi + \subscr{k}{p} h(x) + \subscr{k}{d} J_h(x)\dot{x}\right) \\
\dot{\xi} = \subscr{k}{i} h(x),
\end{cases}
\iff
\begin{cases}
M(x)\dot{x} = - \nabla f(x) - J_h(x)^\top \xi - \subscr{k}{p} J_h(x)^\top h(x)\\
\dot{\xi} = \subscr{k}{i} h(x),
\end{cases}
\eeq
where $\xi\in\R^m$ is the internal (integral) state, and we have defined the positive definite matrix $M(x):= I_n + \subscr{k}{d} J_h(x)^\top J_h(x)$. Note that for $\subscr{k}{d}\geq 0$, the matrix $M(x)$ is always symmetric and positive definite, since $J_h(x)^\top J_h(x) \succeq 0$, for all $x \in \R^n$.
We refer to the dynamics~\eqref{eq:closed_loop_metric} as the \emph{PID saddle-point flow (PID-SPF)}.
In particular, the closed-loop equilibria coincide with the stationary points of the equality-constrained optimization problem.
\begin{mytheorem}[label=thm:unified_pid_framework]{PID controllers generate saddle-point flows of the augmented Lagrangian}
Consider the equality-constrained problem~\eqref{eq:eq_constrained}, where $\map{f}{\R^n}{\R}$ is continuously differentiable and $\map{h}{\R^n}{\R^m}$ is twice continuously differentiable. Given $\subscr{k}{p} \geq0 $, $\subscr{k}{i} > 0$, and $\subscr{k}{d} \geq 0$, consider the dynamics~\eqref{eq:closed_loop_metric}, the augmented Lagrangian
$$
\subscr{L}{aug}(x,\xi) = f(x) + \xi^\top h(x) + \frac{\subscr{k}{p}}{2} \norm{h(x)}^2,
$$
and the metric $M(x) := I_n + \subscr{k}{d} J_h(x)^\top J_h(x)$. Then:
\begin{enumerate}[label=\textup{(\roman*)}]
\item \label{thm_unified:item1}
The equilibria of~\eqref{eq:pid_cmo} coincide with the equilibria of~\eqref{eq:closed_loop_metric} and with the stationary points of~\eqref{eq:eq_constrained};
\item \label{thm_unified:item2}
If $\subscr{k}{d} = 0$, the coordinate transformation
\beq
\label{eq:diffeo}
\map{T}{(x,\lambda)}{(x,\xi) := (x,\lambda - \subscr{k}{p} h(x))}
\eeq
is a smooth global diffeomorphism;
\item \label{thm_unified:item3}
If $\subscr{k}{d} > 0$, the closed-loop dynamics~\eqref{eq:closed_loop_metric} coincide with the Riemannian saddle-point flow
\beq
\label{eq:riemannian_saddle_explicit}
\begin{cases}
\dot x = - M(x)^{-1} \nabla_x \subscr{L}{aug}(x,\xi), \\
\dot\xi = \subscr{k}{i} \nabla_\xi \subscr{L}{aug}(x,\xi).
\end{cases}
\eeq
Equivalently, the primal dynamics correspond to gradient descent of $\subscr{L}{aug}$ under the Riemannian metric induced by $M(x)$, while the dual dynamics correspond to Euclidean gradient ascent.
\end{enumerate}
\end{mytheorem}
\begin{proof}
First, we prove item~\ref{thm_unified:item1}. Let $(x^\star, \lambda^\star)$ be an equilibrium point of the closed-loop dynamics~\eqref{eq:open_system} -  \eqref{eq:pid_controller}. Since $\subscr{k}{i} \neq 0$, the equilibrium conditions $\dot{x}=\0_n$ and $\dot{\lambda}=\0_m$ imply $h(x^\star)=\0_m$. This in turn implies, $\xi^\star := \lambda^\star - \subscr{k}{p} h(x^\star) - \subscr{k}{d} J_h(x^\star)\dot{x}^\star = \lambda^\star$. Therefore $(x^\star, \lambda^\star) = (x^\star, \xi^\star)$ is an equilibrium of~\eqref{eq:closed_loop_metric}. Substituting into~\eqref{eq:closed_loop_metric} yields
$$
\begin{cases}
\nabla f(x^\star) + J_h(x^\star)^\top \lambda^\star + \subscr{k}{p} \nabla \left(\frac{1}{2}\norm{h(x^\star)}^2\right) = \nabla f(x^\star) + J_h(x^\star)^\top \lambda^\star = \0_{n},\\
h(x^\star) = \0_{m}.
\end{cases}
$$
These conditions coincide with the first-order optimality conditions~\eqref{eq:condition_stationary_point}, thus proving the statement.

To prove item~\ref{thm_unified:item2}, note that since $h$ is smooth, the map $T$ is smooth. Its inverse is explicitly given by $T^{-1}(x,\xi) = \bigl(x,\xi + \subscr{k}{p} h(x)\bigr)$, which is also smooth.
Moreover, its Jacobian is given by
\[
J_T(x,\lambda) =
\begin{bmatrix}
I_n & 0 \\
- \subscr{k}{p} J_h(x) & I_m
\end{bmatrix},
\]
which is block triangular with determinant equal to $1$. Hence $J_T(x,\lambda)$ is nonsingular for all $(x,\lambda) \in \R^{n+m}$. Therefore $T$ is a smooth bijection with smooth inverse, and thus a global diffeomorphism.

Finally, we show item~\ref{thm_unified:item3}. Note that $\nabla_x \subscr{L}{aug}(x,\xi) = \nabla f(x) + J_h(x)^\top \xi + \subscr{k}{p} J_h(x)^\top h(x)$ and $\nabla_{\xi} \subscr{L}{aug}(x,\xi) = h(x)$.
Moreover, $M(x)\succ 0$ implies that the equation $M(x)\dot x = -\nabla_x \subscr{L}{aug}(x,\xi)$ admits the unique solution $\dot x = - M(x)^{-1}\nabla_x \subscr{L}{aug}(x,\xi).$ The dual equation follows directly from $\dot \xi = \subscr{k}{i} h(x) = \subscr{k}{i} \nabla_\xi \subscr{L}{aug}(x,\xi)$. This proves that the closed-loop dynamics~\eqref{eq:closed_loop_metric} coincides with the saddle-point flow~\eqref{eq:riemannian_saddle_explicit}.
\smallskip
\end{proof}
For the existence and uniqueness of the trajectories of~\eqref{eq:closed_loop_metric} it suffices that $h$ is continuously differentiable. This assumption guarantees continuity of the induced metric $M(x)$ and local Lipschitz continuity of the vector field. However, to interpret~\eqref{eq:closed_loop_metric} as a Riemannian flow in the classical sense, we require $h$ to be twice continuously differentiable so that the metric tensor $M(x)$ is continuously differentiable. 
Notably, in the absence of the derivative term (i.e., when $\subscr{k}{d}=0$), the metric reduces to the identity and all results in Theorem~\eqref{thm:unified_pid_framework} only require $h$ to be continuously differentiable.

From the perspective of dual variables, it is important to note that $\xi$ does not coincide with the Lagrange multiplier $\lambda$ of problem~\eqref{eq:eq_constrained}. Instead, it represents a shifted version of $\lambda$ and plays the role of the multiplier associated with the augmented Lagrangian. At equilibrium we have $h(x^\star)=\0_m$ and $\dot{x}^\star=\0_n$, so that $\lambda^\star=\xi^\star$. Therefore $\xi$ converges to the true Lagrange multiplier of the constrained problem.

Theorem~\ref{thm:unified_pid_framework}.\ref{thm_unified:item2} establishes a correspondence between the PI-SPF (that is, the dynamics~\eqref{eq:closed_loop_metric} with $\subscr{k}{d}=0$) and the PI-CMO dynamics introduced in~\cite{VC-SMF-SP-DR:25} via the diffeomorphism in~\eqref{eq:diffeo}. In particular, the smooth transformation $T$ defines a relation between the two dynamical systems, preserving trajectories in the primal variable while inducing a smooth reparameterization of the dual variables. Consequently, the two systems share the same asymptotic properties, while their transient behavior and convergence rates may differ due to the distinct evolution of the dual variables.

\begin{rem}[Comparison with~\cite{JR-SLJ:26}]
Recent work~\cite{JR-SLJ:26} adopts an optimization perspective for discrete-time PI control, establishing connections with augmented Lagrangian methods. In contrast, our analysis is formulated in continuous time and characterizes the induced dynamics as a saddle-point flow within a control-theoretic framework. Including the derivative term further generalizes these results, yielding a fundamentally different structure: a Riemannian saddle-point flow with a state-dependent metric.
\end{rem}
Table~\ref{tab:pid_special_cases} summarizes the connections between PID-SPF and other known flows.
\begin{table}[!h]
\centering
\begin{tabular}{c c c c c}
\hline
Controller & $\subscr{k}{i}$ & $\subscr{k}{p}$ & $\subscr{k}{d}$ & Resulting Dynamics \\
\hline\\[-8pt]
I   & $>0$ & $0$ & $0$ & Arrow-Hurwicz-Uzawa flow \\[4pt]

PI  & $>0$ & $>0$ & $0$ & Augmented Lagrangian primal-dual \\[4pt]

PID & $>0$ & $\geq 0$ & $>0$ & Riemannian saddle-point flow \\[4pt]

PID & $>0$ & $\geq 0$ & $\to\infty$ & Projected saddle-point flow \\[2pt]
\hline
\end{tabular}
\caption{Correspondence between PID controller and saddle-point dynamics. The integral gain is always positive to ensure feasibility.}
\label{tab:pid_special_cases}
\end{table}

\subsection{Convergence Analysis of PID-SPF in Convex Settings}
\label{sec:convergence}
We study the convergence properties of PID-SPF in the convex setting with affine constraints. In this case, the Jacobian of the constraint map is constant, which simplifies the geometry of the dynamics and allows us to establish global exponential convergence using contraction theory.
Consider $h(x) := Ax - b$, where $A \in \R^{m \times n}$, $b \in \R^m$.
The equality-constrained optimization problem then becomes
\begin{align*}
\min_{x \in \R^n} & f(x)\\
\text{s.t. } & Ax = b.
\end{align*}
The PID-SPF~\eqref{eq:riemannian_saddle_explicit} becomes
\beq
\label{eq:riemannian_saddle_affine}
\begin{cases}
\dot x = {-} \left(I_n {+} \subscr{k}{d} A^\top A\right)^{-1} \left(\nabla f(x) {+} A^\top \xi {+} \subscr{k}{p} A^\top(Ax {-} b)\right), \\
\dot\xi = \subscr{k}{i} \left(Ax - b\right).
\end{cases}
\eeq
We let $z = (x,\xi) \in \R^{n + m}$ and  $\map{F}{\R^{n+m}}{\R^{n+m}}$ be the vector field~\eqref{eq:riemannian_saddle_affine} for $\dot{z} = F(z)$. 
We work under the following assumptions on the function $f$ and the matrix $A$.
\begin{myassumptionstar}{Convexity and Constraint Structure}
For the dynamics~\eqref{eq:riemannian_saddle_affine}, assume:
\begin{enumerate}[label=\textup{($A$\arabic*)}, leftmargin=3em]
    \item \label{ass:1_f} the function $\map{f}{\R^n}{\R}$ is $\rho$-strongly convex and $L$-smooth;
    \item \label{ass:2_A} 
    the matrix $A \in \R^{m \times n}$ satisfies $\amin I_m \preceq AA^\top \preceq \amax I_m$ for $\amin,\amax \in \R_{>0}$.
\end{enumerate}
\end{myassumptionstar}
\smallskip
Assumption~\ref{ass:2_A} implies that $A$ has full row rank and, in particular, that $m \leq n$. Under Assumptions~\ref{ass:1_f} and~\ref{ass:2_A}, problem~\eqref{eq:eq_constrained} admits a unique global minimum. These assumptions are standard in the literature when establishing global convergence to the equilibrium (see, e.g.,~\cite{GQ-NL:19, AD-VC-AG-GR-FB:23f, VC-SMF-SP-DR:25}).
The next theorem shows that, under this Assumption, the PID-SPF~\eqref{eq:riemannian_saddle_affine} is globally contracting.
\begin{mytheorem}[label=thm:riemannian_saddle_affine-contractivity]{Contractivity of~\eqref{eq:riemannian_saddle_affine}}
Under Assumptions~\ref{ass:1_f} and~\ref{ass:2_A}, for any $\subscr{k}{p} \geq0 $, $\subscr{k}{i} > 0$, and $\subscr{k}{d} \geq 0$, the PID-SPF~\eqref{eq:riemannian_saddle_affine} is strongly infinitesimally contracting with respect to $\|\cdot\|_{P}$ with rate $c > 0$ where
\begin{align}
P &=
\begin{bmatrix}
M & \alpha A^\top \\
\alpha A & \tau I_m
\end{bmatrix}
\succ 0,  \label{thm:def:P}\\
\alpha&=\frac{1}{2}\min\Big\{\frac{1}{L + \subscr{k}{p}\amax}, \subscr{k}{i}^{-1}\frac{\rho}{\amax}\Big\}, \label{thm:def:alpha} \\
c&=\frac{1}{2}\alpha\subscr{k}{i}\frac{\amin}{1 + \subscr{k}{d}\amax} = \frac{1}{4} \min\Big\{\subscr{k}{i}\frac{1}{L + \subscr{k}{p}\amax} \frac{\amin}{1 + \subscr{k}{d}\amax},\frac{\amin}{\amax}\frac{\rho}{1 + \subscr{k}{d}\amax}\Big\}. \label{thm:def:c}
\end{align}
\end{mytheorem}
\begin{proof}
By Assumption~\ref{ass:1_f}, $f$ is differentiable almost everywhere by Rademacher's theorem, and the Jacobian of the dynamics~\eqref{eq:riemannian_saddle_affine} exists almost everywhere and is given by
$$
\subscr{J}{F}(z)
:=
\begin{bmatrix}
- M^{-1} \left(\nabla^2 f(x) + \subscr{k}{p} A^\top A\right) & - M^{-1} A^\top\\
\subscr{k}{i}A & 0 
\end{bmatrix}.
$$
To prove strong infinitesimal contractivity it suffices to show that for all $z$ for which $\subscr{J}{F}(z)$ exists, the bound $\mu_{P}(\subscr{J}{F}(z)) \leq -c$ holds for $P, c$ given in~\eqref{thm:def:P},~\eqref{thm:def:alpha} and~\eqref{thm:def:c}, respectively. The assumption of $\rho$-strong convexity and $L$-smoothness of $f$ imply the inequalities $\rho I_n \preceq \nabla^2 f(x) \preceq L I_n$, for all $x$ for which the Hessian exists. Moreover, it holds:
\begin{equation*}
\sup_{z}\mu_{P}(\subscr{J}{F}(z)) \leq \max_{\rho I_n \preceq B \preceq L I_n} \mu_{P}\left(
\begin{bmatrix}
- M^{-1} \left(B + \subscr{k}{p} A^\top A\right)  & - M^{-1} A^\top \\
\subscr{k}{i}A & 0
\end{bmatrix}
\right), 
\end{equation*}
where the $\sup$ is over all points for which $\subscr{J}{F}(z)$ exists. 
The above matrix has the general scaled saddle structure of Lemma~\ref{lemma:saddle-matrices_general} in the Appendix with $B' := B + \subscr{k}{p} A^\top A$, $\tau := \subscr{k}{i}^{-1}$, and $M = I_n + \subscr{k}{d}A^\top A$. Assumptions~\ref{ass:1_f} and~\ref{ass:2_A} implies that $\rho \preceq B' \preceq (L + \subscr{k}{p}\amax) I_n$, and $I_n \preceq M \preceq (1 + \subscr{k}{d}\amax) I_n$.
The result then follows from Lemma~\ref{lemma:saddle-matrices_general} with $\subscr{m}{min}= 1$, $\subscr{m}{max}= 1 + \subscr{k}{d} \amax$, $\subscr{b}{min} = \rho$, and $\subscr{b}{max} = L + \subscr{k}{p}\amax$.
\end{proof}
Theorem~\ref{thm:riemannian_saddle_affine-contractivity} proves that~\eqref{eq:riemannian_saddle_affine} is contracting for any gains $\subscr{k}{i}>0$, $\subscr{k}{p}\geq 0$, and $\subscr{k}{d}\geq0$. In particular, no additional tuning conditions are required to guarantee global exponential convergence. The gains only affect the contraction rate $c$ in~\eqref{thm:def:c}, and therefore influence the speed of convergence but not stability. Consequently, all trajectories of~\eqref{eq:riemannian_saddle_affine} converge exponentially to a unique equilibrium. Contractivity also ensures incremental stability and robustness to vector-field perturbations~\cite{FB:26-CTDS_cdc}.

Let $(x(t), \xi(t))$ and $(x^\star, \xi^\star)$ be a trajectory and the unique equilibrium of~\eqref{eq:riemannian_saddle_affine}, respectively, and define $\delta x := x(t) - x^\star$ and $\delta \xi := \xi(t) - \xi^\star$. The matrix $P$ in Theorem~\ref{thm:riemannian_saddle_affine-contractivity} defines the Lyapunov function
$$
V(x,\xi) =
\begin{bmatrix} \delta x \\ \delta \xi\end{bmatrix}^\top
P
\begin{bmatrix} \delta x \\ \delta \xi\end{bmatrix}
= \delta x^\top M \delta x + 2\alpha \delta x^\top A^\top \delta \xi + \subscr{k}{i}^{-1} \|\delta \xi\|^2.
$$
By Theorem~\ref{thm:riemannian_saddle_affine-contractivity}, this is an energy that decays exponentially along the trajectories of PID-SPF.
The terms of $V$ reflect the structure of the dynamics: $\delta x^\top M \delta x$ captures the Riemannian geometry induced by $\subscr{k}{d}$, $\subscr{k}{i}^{-1} \|\delta \xi\|^2$ measures the dual energy, and the cross-term $2\alpha \delta x^\top A^\top \delta \xi$ represents the coupling between primal and dual variables induced by the constraints.

Using Euler discretization with sufficiently small step sizes~\cite{FB-PCV-AD-SJ:21e}, convergence guarantees are preserved for the discrete-time system corresponding to~\eqref{eq:riemannian_saddle_affine}, providing a computationally efficient numerical scheme for computing the equilibrium of the dynamics.
\section{Numerical Examples}
\label{sec:numerical_example}
We illustrate the effectiveness of PID-SPF~\eqref{eq:closed_loop_metric} in solving equality-constrained OPs~\eqref{eq:eq_constrained} via (i) a quadratic program with linear constraints, and (ii) a bilevel optimization problem.

\subsection{Quadratic Programming}
Consider problem~\eqref{eq:eq_constrained} with objective function $f(x) = x^\top Qx$ and constraints $h(x) = Ax-b$. The cost function is chosen to satisfy Assumption~\ref{ass:1_f}. Specifically, the matrix $Q$ is randomly generated such that the resulting quadratic function is strongly convex and smooth with parameters $\rho = 3$ and $L = 4$, respectively. We randomly select $A$ satisfying Assumption~\ref{ass:2_A}.
We set $n=10$, $m= 2$, $\subscr{k}{p} = 15$ and $\subscr{k}{i} = 100$, while we consider different values of $\subscr{k}{d}$ to illustrate its effect on convergence. We simulate~\eqref{eq:riemannian_saddle_affine} over the time interval $t\in [0,20]$ with a forward Euler discretization with stepsize $\Delta t =0.01$ that is in accordance with the conditions in~\cite{FB-PCV-AD-SJ:21e}, starting from random initial conditions sampled from a uniform distribution over $[0,2)$. 
Figure~\ref{fig:QPRes} illustrates the mean and the minimum and maximum values of the logarithm of the $\norm{\cdot}_P$ distance from the optimal solution $z^\star$ across 50 simulated trajectories of~\eqref{eq:riemannian_saddle_affine}. The solution to the equality constrained quadratic program is computed explicitly. In agreement with Theorem~\ref{thm:riemannian_saddle_affine-contractivity}, the figure shows that convergence is linearly bounded.
Moreover, consistent with the rate characterized in Theorem~\ref{thm:riemannian_saddle_affine-contractivity}, the results show how increasing $\subscr{k}{d}$ may decrease the convergence rate, depending on the parameters of the OP.
The PI-CMO dynamics are related to~\eqref{eq:riemannian_saddle_affine} through the smooth diffeomorphism~\eqref{eq:diffeo}, so they produce the same trajectories only in the primal variable $x$, while the dual variables are transformed by $T$.
We do not include a comparison with the PI-CMO convergence rate in~\cite{VC-SMF-SP-DR:25}, since its preference over the equivalent controller (setting $\subscr{k}{d} = 0$ in PID-SPF) is a function of the OP parameters as well.
\begin{figure}[ht]
    \centering
\includegraphics[width=.55\linewidth,page = 1]{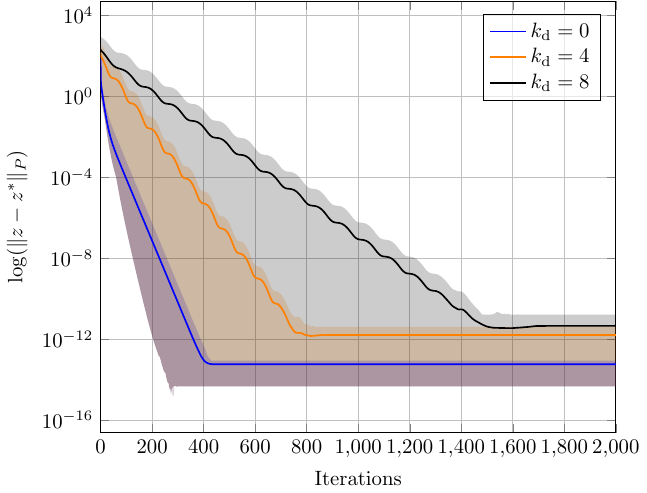}
    \caption{Optimization of a  quadratic program with linear constraints using PID-SPF for $n=10$, $m= 2$. The PID gains are set as $\subscr{k}{p} = 15$, $\subscr{k}{i} = 100$ and $\subscr{k}{d} \in \{0, 4, 8\}$. The mean across $50$ different initial conditions is plotted in the figure and the shaded region corresponds to the minimum and maximum values of $\log(\|z-z^\star\|_P)$ at each iteration.}
    \label{fig:QPRes}
\end{figure}
\subsection{Bilevel Optimization}
We consider bilevel optimization problems of the form
\begin{equation}
\begin{aligned}
\min_{x \in \R^n} & f\bigl(x, y^\star(x)\bigr)\\
\text{s.t. } & y^\star(x) \in \argmin_{y \in \R^m} g(x,y),
\end{aligned}
\label{op:bilevel-opt}
\end{equation}
where the functions $\map{f}{\R^{n}\times\R^m}{\R}$ and $\map{g}{\R^{n}\times\R^m}{\R}$ are smooth. Bilevel optimization is commonly used to model Stackelberg games. Here, we consider a specific choice of $f(x,y)$ and $g(x,y)$ that emulate a leader-follower game for risk-aware resource allocation. The leader sets an allocation vector $x \in \mathbb{R}^n$ using a smooth penalization of the worst-case via a log-sum-exp term in the upper-level objective. The map $f(x,y)$ also includes a consistency term enforcing agreement between the anticipated response $Cx$ and the actual system reaction $y(x)$. The follower responds accordingly, and sets an operational state $y \in \mathbb{R}^m$ by minimizing a quadratic cost.
The convex upper-level objective function in~\eqref{op:bilevel-opt} is therefore
\begin{equation*}
    f(x,y) = \log\left(\sum_{i=1}^n \e^{x_i}\right)+\lambda\|Cx-y(x)\|_2^2,
\end{equation*}
while the quadratic lower-level cost function is
\begin{equation*}
    g(x,y) = \frac{1}{2}y^\top Qy+(A x+b)^\top y,
\end{equation*}
where $Q = Q^\top \in \R^{m \times m}$, $Q \succ 0$, $A \in \R^{m \times n}$ and $b \in \R^m$. The gradient of the lower-level objective with respect to $y$ is $h(x,y) := \nabla_y g(x,y) = A x + Q y + b$, and can be rewritten as $h(x,y) = A' [x,y]^\top - b$, where $A^\prime := [A \quad Q]$.
Using the first-order optimality condition for the lower-level problem, we obtain the following equivalent constrained optimization problem in the form of~\eqref{eq:eq_constrained} 
\beq
\label{op:bilevel_constrained}
\begin{aligned}
\min_{x \in \R^n, y \in \R^m} & f(x,y)\\
\text{s.t. } & h(x,y) := \nabla_y g(x,y) = \0_m.
\end{aligned}
\eeq
Bilevel optimization algorithms theoretically require the exact lower-level optimum. In practice, however, numerical solvers can only provide approximate solutions within a finite number of iterations, thereby introducing uncertainty into the upper-level optimization. To capture this effect, we consider the first-order optimality condition in~\eqref{op:bilevel_constrained} in the presence of bounded noise. Specifically, we introduce $\tilde h(x,y) = h(x,y)+w$ in~\eqref{eq:closed_loop_metric} instead of $h(x,y)$, where $w \in \mathbb{R}^m$ satisfies $\|w\|_2\leq W$.

We simulate PID-SPF~\eqref{eq:riemannian_saddle_affine} to solve~\eqref{op:bilevel_constrained} over the interval $t\in[0,20]$ using forward Euler discretization with time step $\Delta t = 0.01$. Since the problem formulation does not fulfill Assumption~\ref{ass:1_f}, the step size is chosen arbitrarily, with $\Delta t \in (0,1)$.  The parameters $Q \succ 0$, $A$, $b$, $C$ and $\lambda$ are chosen arbitrarily. Also, we set $n=1$, $m=1$, $\lambda = 0.01$, $\subscr{k}{p} = 15$, $\subscr{k}{i} = 100$, and we consider various values of $\subscr{k}{d}$.

For an injected uncertainty level of $W = 0.5$, Figure~\ref{fig:BilevelRes} shows the convergence of PID-SPF for different values of $\subscr{k}{d}$, along with the corresponding trajectory towards the solution of the bilevel problem $(x^\star,y^\star)$. This solution is computed via the constrained trust region method in \texttt{SciPy}.
For the chosen $\Delta t$, the optimization algorithm does not converge for $\subscr{k}{d} = 0.0$, which further emphasizes the role played by the derivative term in projecting onto the solution set of the lower-level problem. In the presence of uncertainty, increasing the value of $\subscr{k}{d}$ results in 
converging to a smaller neighborhood of the optimal solution, in agreement with our analysis of convergence of PID-SPF towards projected saddle point flow as $\subscr{k}{d} \to \infty$. Moreover, as in the control of general linear systems, increasing $\subscr{k}{d}$ decreases the overshoot, corresponding to a dampened oscillatory behavior toward $(x^\star,y^\star)$ in the optimization setting.
\begin{figure}[h]
    \centering
    \includegraphics[width = .55\linewidth, page = 2]{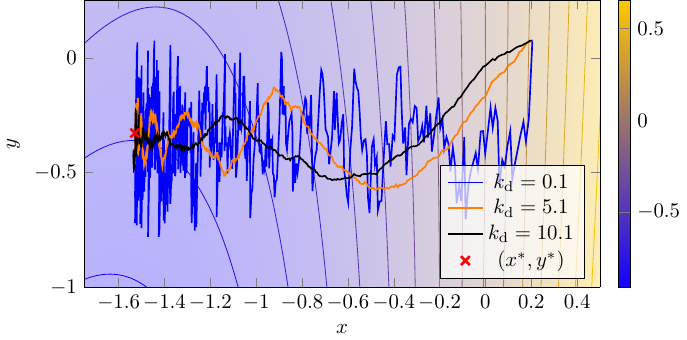}
    \caption{Bilevel Optimization using PID-SPF for $\subscr{k}{d} \in \{0.1, 5.1, 10.1\}$, $\subscr{k}{p} = 15$, and $\subscr{k}{i} = 100$. Contour plot is that of $f(x,y)$ for $n=1$, and $m=1$.}
    \label{fig:BilevelRes}
\end{figure}

Ultimately, the above results extend to bilevel problems with upper-level cost $f(x,y)$ and strongly convex, twice differentiable lower-level objective $g(x,y)$. The lower-level problem then has a unique minimizer for every $x\in\R^n$, and its first-order optimality condition is necessary and sufficient. Therefore, the bilevel problem can be equivalently written as~\eqref{op:bilevel_constrained} and solved accordingly using~\eqref{eq:riemannian_saddle_explicit}.
\section{Discussion and Conclusion}
\label{sec:discussion_conclusion}
We presented a unified control-theoretic framework for saddle-point flows for equality-constrained optimization.
We showed that PID feedback on the dual variable systematically generates a class of saddle-point flows associated with the (augmented) Lagrangian. We termed the resulting dynamics~\eqref{eq:closed_loop_metric} \emph{PID-saddle-point flow} (PID-SPF) and showed that its equilibria coincide with the stationary points of the original OP~\eqref{eq:eq_constrained}.
Moreover, we characterized how the feedback affects the resulting dynamics.
Then, for strongly convex and $L$-smooth cost functions, and full row-rank affine constraints, we proved strong infinitesimally contractivity of the PID-SPF and derived explicit bounds on the convergence rate.
Finally, we illustrated the effectiveness of our framework on equality-constrained quadratic programs and on a bilevel optimization problem with a strongly convex lower-level objective corresponding to uncertain optimality conditions.

As future work, it would be of interest to (i) extend the convergence analysis to general nonlinear constraints, including inequality constraints, and nonconvex objectives, thereby broadening the applicability of our framework; (ii) analyze the discretization of the continuous-time flow and the properties of the resulting algorithms; 
(iii) provide convergence analysis in the presence of uncertainty; and (iv) further investigate the role of the geometry induced by the derivative action, as well as its potential connections with adaptive and preconditioned optimization methods.

\section*{Acknowledgments}
The authors wish to thank Jonas Ohnemus for insightful comments.

This work was supported as a part of NCCR Automation, a National Centre of Competence in Research, funded by the Swiss National Science Foundation (grant number 51NF40\_225155)

\bibliographystyle{plainurl+isbn}
\bibliography{main}

\appendix
\renewcommand{\thelem}{\Alph{section}.\arabic{lem}}
\section{Appendix}
\label{apx:log_norm_of_saddle_matrix}
This section derives explicit bounds on the logarithmic norm of the scaled saddle matrices that appear as Jacobians of the PID-SPF~\eqref{eq:riemannian_saddle_affine}.
The results extend the logarithmic norm analysis of saddle matrices appearing in primal-dual dynamics in~\cite[Lemma 20]{AD-VC-AG-GR-FB:23f} and recover that result then when $\subscr{k}{p}= \subscr{k}{d} = 0$.
\begin{lem}[Logarithmic norm of scaled saddle matrices]
\label{lemma:saddle-matrices_general}
Given $B=B^\top$, $M = M^\top \in\R^{n \times n}$, $M\succ 0$, and $A\in\R^{m \times n}$, with $m\leq n$, we consider the \emph{scaled saddle matrix}
\begin{equation*}
S =
\begin{bmatrix}
-M^{-1}B & - M^{-1}A^\top \\
\tau^{-1} A & 0
\end{bmatrix}
\in \R^{(m+n)\times (m+n)}.
\end{equation*}  
Then, for each $(M, B,A)$ satisfying $\subscr{m}{min}I_n\preceq M \preceq \subscr{m}{max} I_n$, $\subscr{b}{min} I_n\preceq B \preceq \subscr{b}{max} I_n$ and $\amin I_m \preceq A A^\top \preceq \amax I_m$, for $\subscr{m}{min}, \subscr{m}{max}, \subscr{b}{min}, \subscr{b}{max}, \amin,\amax \in \R_{>0}$, the following contractivity LMI holds: $S^\top P + PS\preceq -2 c P \iff  \lognorm{S}{P} \leq - c$, where
\begin{align}
P &=
\begin{bmatrix}
M & \alpha A^\top \\
\alpha A & \tau I_m
\end{bmatrix}
\succ 0, \label{lem:def:P}\\
\alpha&=\frac{1}{2}\min\Big\{\frac{\subscr{m}{min}}{\subscr{b}{max}},\tau\frac{\subscr{b}{min}}{\amax}\Big\}, 
\label{lem:def:alpha} \\
c&=\frac{1}{2}\alpha\tau^{-1}\frac{\amin}{\subscr{m}{max}}.
\label{lem:def:c}
\end{align}
\end{lem}
\begin{proof}
\textbf{Step 1: Positivity of P.}\\
We start by verifying that $P\succ0$. Using the Schur complement of the $(2,2)$ entry, $P\succ0$ is equivalent to
$
M - \tau^{-1} \alpha^2 A^\top A\succ 0.
$
For this LMI holds it suffices that
\begin{align*}
\subscr{\lambda}{min}(M) > \tau^{-1} \alpha^2 \subscr{\lambda}{max}(A^\top A) \iff  \alpha^2 < \tau \frac{\subscr{m}{min}}{\amax}.
\end{align*}
The bound $\ds \alpha^2 < \tau \frac{\subscr{m}{min}}{\amax}$ follows from the tighter inequality $(2\alpha)^2 \leq \displaystyle \tau \frac{\subscr{m}{min}}{\amax}$ which is proved as follows:
\begin{align*}
(2\alpha)^2 = \min\Big\{\frac{\subscr{m}{min}}{\subscr{b}{max}},\tau\frac{\subscr{b}{min}}{\amax}\Big\}^2&\leq
\min\Big\{\frac{\subscr{m}{min}}{\subscr{b}{max}},\tau\frac{\subscr{b}{min}}{\amax}\Big\} \cdot
\max\Big\{\frac{\subscr{m}{min}}{\subscr{b}{max}},\tau\frac{\subscr{b}{min}}{\amax}\Big\}\\
&= \frac{\subscr{m}{min}}{\subscr{b}{max}} \cdot \tau\frac{\subscr{b}{min}}{\amax} \leq \tau \frac{\subscr{m}{min}}{\amax}.
\end{align*}
\textbf{Step 2: Expansion of the LMI.}\\
Next, we aim to show that $-S^\top P - PS -2 c P\succeq 0$. We compute
\begin{align}
-  S^\top P&  - PS - 2 c P
=
\begin{bmatrix}
BM^{-1} & - \tau^{-1}A^\top \\
A M^{-1} & 0
\end{bmatrix}
\begin{bmatrix}
M & \alpha A^\top \\
\alpha A & \tau I_m
\end{bmatrix}
+
\begin{bmatrix}
M & \alpha A^\top \\
\alpha A & \tau I_m
\end{bmatrix}
\begin{bmatrix}
M^{-1}B & M^{-1}A^\top \\
-\tau^{-1} A & 0
\end{bmatrix}
- 2 c
\begin{bmatrix}
M  & \alpha A^\top \\
\alpha A & \tau I_m
\end{bmatrix}
\\
&= 
\begin{bmatrix}
BM^{-1}M  + M M^{-1}B  - \tau^{-1} \alpha  A^\top A  - \tau^{-1} \alpha A^\top A  - 2 c M 
&\alpha B M^{-1}  A^\top -  A^\top + M M^{-1} A^\top - 2 c \alpha  A^\top\\
A M^{-1}M + \alpha AM^{-1}B -  A - 2 c \alpha A
&  2 \alpha A M^{-1} A^\top  - 2 c \tau I_m
\end{bmatrix}\\
&= 
\begin{bmatrix}
2 B - 2 \tau^{-1} \alpha  A^\top A  - 2 c M 
& \left(\alpha BM^{-1} - 2 c \alpha  I_n\right)A^\top \\
A \left(\alpha M^{-1}B - 2 c \alpha I_n\right)
&  2 \alpha A M^{-1}A^\top  - 2 \tau c I_m
\end{bmatrix}.
\end{align}
The (2,2) block satisfies the lower bound 
\begin{align*}
2 \alpha A M^{-1}A^\top  - 2 \tau c I_m
&= 2 \big(\frac{1}{2} \alpha A M^{-1} A^\top - \tau c I_m\big) + \alpha A M^{-1} A^\top \\
& \succeq 2\big(\alpha\frac{1}{2}\frac{\amin}{\subscr{m}{max}} - \tau c \big)I_m   + \alpha A M^{-1} A^\top 
\overset{\eqref{lem:def:c}}{=} \alpha A M^{-1} A^\top\succ 0.
\end{align*}
Given this lower bound, we can factorize the resulting matrix as follows:
$$
-S^\top P - PS -2 c P \succeq
\begin{bmatrix} I_n & 0 \\ 0 & A\end{bmatrix}
\underbrace{\begin{bmatrix}
2 B - 2 \tau^{-1} \alpha  A^\top A  - 2 c M 
&\alpha BM^{-1} - 2 c \alpha  I_n\\
\alpha M^{-1}B - 2 c \alpha I_n
& \alpha M^{-1}
\end{bmatrix}}_{2n\times 2n}
\begin{bmatrix} I_n & 0 \\ 0 & A^\top \end{bmatrix}.
$$
\textbf{Step 3: Schur complement verification.}\\
Since $\alpha M^{-1} \succ 0$, it now suffices to show that the Schur complement of the (2,2) block of the $2n\times 2n$ matrix is positive
semidefinite. We compute
\begin{align*}      
& 2 B - 2 \tau^{-1} \alpha  A^\top A  - 2 c M - \alpha \bigl(BM^{-1} - 2 c  I_n\bigr) M 
\bigl(M^{-1}B - 2 c I_n\bigr) \succeq 0 \\
& \iff  2 B - 2 \tau^{-1} \alpha  A^\top A  - 2 c M - \alpha \bigl(B - 2 c M\bigr)\bigl(M^{-1}B - 2 c I_n\bigr) \succeq 0\\
& \iff  2 B - 2 \tau^{-1} \alpha  A^\top A  - 2 c M - \alpha \bigl(BM^{-1}B - 2 c B - 2 c B + 4 c^2 M\bigr) \succeq 0\\
&\impliedby 2B  - \alpha BM^{-1}B \succeq 2 \tau^{-1} \alpha  A^\top A  + 2 c M
\quad \text{and} \quad 4 \alpha c B \succeq 4 \alpha c^2 M.
\end{align*}

\textbf{Step 4: Prove $\mathbf{2B - \alpha BM^{-1}B \succeq 2 \tau^{-1} \alpha  A^\top A  + 2 c M}$.}\\
To prove
\beq
\label{eq:ineq_1_lem}
2B - \alpha BM^{-1}B \succeq 2 \tau^{-1} \alpha  A^\top A  + 2 c M,
\eeq
we first upper bound its right hand side as follows
\begin{align*}
2 \tau^{-1} \alpha  A^\top A  + 2 c M &\preceq \alpha \left(2 \tau^{-1} \amax + 2 c \subscr{m}{max}\right) \overset{\eqref{lem:def:c}}{\preceq} \alpha\tau^{-1} (2\amax + \amin)I_n \\
& \!\!\!\!\!\!\!\!\! \overset{\alpha\leq\frac{1}{2}\subscr{b}{min}\tau/\amax}{\preceq}
\frac{1}{2} \frac{\subscr{b}{min}}{\amax} (2\amax + \amin) I_n \preceq \ds \frac{3}{2}\subscr{b}{min} I_n.
\end{align*}
Next, we upper bound the left hand side of~\eqref{eq:ineq_1_lem} as follows:
\begin{align*}
2B - \alpha BM^{-1}B &\succeq 2B - \alpha\frac{1}{\subscr{m}{min}} B^2 \succeq 2B - \alpha\frac{\subscr{b}{max}}{\subscr{m}{min}} B \succeq \left(2 - \alpha\frac{\subscr{b}{max}}{\subscr{m}{min}}\right) B \succeq \left(2 - \frac{1}{2}\right) B \succeq \frac{3}{2}\subscr{b}{min} I_n.
\end{align*}

\textbf{Step 5: Prove $\mathbf{4 \alpha c B \succeq 4 \alpha c^2 M}$.}\\
The LMI $4 \alpha c B \succeq 4 \alpha c^2 M$ follows from the inequality $\ds c \leq \frac{\subscr{b}{min}}{\subscr{m}{max}}$. This follows noticing that,~\eqref{lem:def:c} implies $\ds c\leq \frac{1}{4} \frac{\amin \subscr{b}{min}}{\amax\subscr{m}{max}} < \frac{\subscr{b}{min}}{\subscr{m}{max}}$.
This concludes the proof.
\end{proof}
\end{document}